\documentclass[11pt, reqno]{amsart}
    \usepackage{a4wide}
    \usepackage{wrapfig}
    \usepackage[normalem]{ulem}
    \newcommand{\stkout}[1]{\ifmmode\text{\sout{\ensuremath{#1}}}\else\sout{#1}\fi}
    \usepackage{amsmath}
    \usepackage{amsthm}
    \usepackage{mathrsfs}
    \usepackage{amssymb}
    \usepackage{enumerate}
    \usepackage{multicol}
    \usepackage{color}
    \usepackage{graphicx}
    \usepackage[table]{xcolor}
    \usepackage{url}
    \usepackage{fancyhdr}
    \usepackage{tikz}
    \usetikzlibrary{matrix}
    \usetikzlibrary{arrows}
    \usepackage{subfig}
    \usepackage{hyperref}
    \usepackage{varwidth}
    \def\id{{\fontsize{.85em}{1.1em}\selectfont1}\normalfont\kern-.8ex1}
    \numberwithin{equation}{section}
    \newcommand*\diff{\mathop{}\!\mathrm{d}}
    
    \newcommand\tenq[2][1]{%
    \def\useanchorwidth{T}%
    \ifnum#1>1%
    \stackunder[0pt]{\tenq[\numexpr#1-1\relax]{#2}}{\scriptscriptstyle\widetilde}%
    \else%
    \stackunder[1pt]{#2}{\scriptscriptstyle\widetilde}%
    \fi%
    }

    \usetikzlibrary{shapes.misc}
    
    \newcommand{\strikeout}[1]{%
    \ifmmode%
          \tikz[inner sep=0.5pt,baseline] \node [strike out,draw=OrangeRed,anchor=text]{$#1$};%
    \else%
          \tikz[inner sep=0.5pt,baseline] \node [strike out,draw=OrangeRed,anchor=text]{#1};%
    \fi%
    }

    \newtheorem{theorem}{Theorem}[section]
    \newtheorem{proposition}[theorem]{Proposition}
    \newtheorem{corollary}[theorem]{Corollary}
    \newtheorem{lemma}[theorem]{Lemma}
    \newtheorem{remark}[theorem]{Remark}
    \newtheorem{definition}[theorem]{Definition}
    \newcommand{\be}{\begin{equation}}
    \newcommand{\ee}{\end{equation}}
    \newcommand{\bp}{\begin{proof}}
    \newcommand{\ep}{\end{proof}}
    \newcommand{\bel}{\begin{equation}\label}
    \newcommand{\eeq}{\end{equation}}
    \newcommand{\bea}{\begin{eqnarray}}
    \newcommand{\eea}{\end{eqnarray}}
    \newcommand{\bee}{\begin{eqnarray*}}
    \newcommand{\eee}{\end{eqnarray*}}
    \newcommand{\ben}{\begin{enumerate}}
    \newcommand{\een}{\end{enumerate}}
    \newcommand{\p}{\partial}
    
    \newcommand{\R}{{\mathbb R}}
    
    \newcommand{\C}{{\mathbb C}}

    \newcommand{\sech}{\operatorname{sech}}

    \title[]{On the nonexistence of NLS breathers}

    \author[Miguel \'A. Alejo]{Miguel \'A. Alejo}
    \address{Departamento de Matem\'aticas. Universidad de C\'ordoba\\
    C\'ordoba, Spain.}
    \email{malejo@uco.es}

    \author[A.J. Corcho]{ Ad\'an J. Corcho}
    \address{Departamento de Matem\'aticas. Universidad de C\'ordoba\\
    C\'ordoba, Spain.}
    \email{a.corcho@uco.es}

    \thanks{M.\'A. Alejo  has been partially supported by Grant PID2022-137228OB-I00 funded by the Spanish Ministerio de Ciencia, Innovaci\'on y Universidades, MICIU/AEI/10.13039/501100011033.}
    
    \date{\today}

\begin{document}

    \maketitle \markboth{On the nonexistence of NLS breathers} {M.\'A. Alejo  and A.J. Corcho}

    \setcounter{page}{1}

    \begin{quote}
    \textbf{Abstract.}
    {\small In this work, a rigorous proof of the nonexistence of breather solutions for NLS equations is presented. By using suitable virial functionals, we are able to characterize the nonexistence of breather solutions, different from standing waves, by only using their inner energy and the power of the corresponding nonlinearity of the equation. We extend this result for several NLS models with different power nonlinearities and even the derivative and logarithmic NLS equations.}
    \end{quote}
    \tableofcontents

    \section{Introduction}\label{Sec1}
    In this work we consider the nonlinear Schr\"odinger equation (NLS),

    \be\label{nNLS}
    iu_t +\Delta u=\varepsilon|u|^pu,\quad (t,x)\in \R\times\R^n,
    \ee

    \medskip 
    \noindent 
     where $u= u(t,x)$ is a complex-valued function, $\varepsilon=\pm1$ and $0<p <p_n^*$ with $p_n^{*}$ denoting the energy critical power, defined by
     \be\label{Pcritico}
     p_n^*=
     \begin{cases}
     \;\,\infty,& n=1,2,\medskip\\
     \frac{4}{n-2},&n\ge3.
     \end{cases}
     \ee
     The cases $\varepsilon =1$ and $\varepsilon=-1$ are called \emph{defocusing} (repulsive nonlinearity) and \emph{focusing} (attractive nonlinearity), respectively. 
     For this model, we will consider two different boundary value conditions (BC) at infinity (i.e. \emph{backgrounds})

    \begin{enumerate}
     \item[(1.1a)] Zero BC: $|u(t,x)|\rightarrow0\quad\text{when}\quad x\rightarrow\pm\infty,$ 
     \item[(1.1b)] Nonzero BC, as a \emph{Stokes wave}: for all $t\in\R,$ 
     \[
    % \be\label{stokesBC}
     |u(t,x) - e^{-i\varepsilon t}|\rightarrow0\quad\text{when}\quad x\rightarrow\pm\infty.
   %  \ee
   \]
    \end{enumerate} 

    \medskip
    Low order conserved quantities for \eqref{nNLS} are, the \emph{mass}:
    \be\label{mass}
    M[u](t):=\|u(t,\cdot)\|^2_{L^2}=M[u](0),
    \ee
    the \emph{energy}:
    \be\label{energy}
    E[u](t):=\|\nabla_xu(t,\cdot)\|^2_{L^2}+\frac{2\varepsilon}{p+2}\|u(t,\cdot)\|^{p+2}_{L^{p+2}}=E[u](0),
    \ee
    and the \emph{momentum}:
    \be\label{momentum}
    P[u](t):=\text{Im}\int_{\R^n}\nabla_xu(t,\cdot)\bar{u}(t,x)dx=P[u](0).
    \ee 

    \medskip 
    The initial value problem (IVP)  for \eqref{nNLS} is, in the mass-subcritical case ($0<p<4/n$), local and globally well posed for initial data in $L^2(\R^n)$. In the mass-critical case ($p=4/n$) local well posedness can also hold for data in $L^2(\R^n)$. In the super-critical case ($p>4/n$), sharp results were obtained for small data in $H^s(\R^n),~~s\leq s_c=\frac{n}{2}-\frac{2}{p}$ \cite{Caze}, \cite{CazeWe}, \cite{CazeWe2},\cite{Dod1}. See \cite{LP} for further reading and a comprehensive description of the initial value problem for \eqref{nNLS}.

    \medskip
    From the physical point of view, this equation appears as a model in several contexts (see for instance: \cite{GV1},\cite{GV2},\cite{SCMc},\cite{ZS}). Assuming $\varepsilon =-1$ (focusing case) and with respect to explicit solutions, beside phases $e^{it},$ the next simplest ones are the so called standing wave solutions. Specifically, standing waves  of \eqref{nNLS}  with $\varepsilon =-1$ are written as
    \[
    u_{\omega}(t,x)=e^{i\omega  t}Q_{\omega}(x),\quad \omega \in \R.\]
    \noindent
    In this case, $Q_{\omega}$  satisfies the nonlinear elliptic equation  
    \be\label{nNLS-elliptic-eq-omega}
    -\Delta Q_{\omega} + \omega Q_{\omega} -|Q_{\omega}|^pQ_{\omega} =0,\quad x\in \R^n,
    \ee
    which has solution for $\omega >0$. Furthermore, $u_{\omega}$ can be seen as the generalization (with $\lambda = \sqrt{\omega}$) of the solution $U(t,x):=e^{it}Q(x)$ with 
    $Q$ satisfying the equation \eqref{nNLS-elliptic-eq-omega} with $\omega=1$, that is, 
    \be\label{nNLS-elliptic-eq-GS}
    -\Delta Q + Q -|Q|^pQ =0,\quad x\in \R^n.
    \ee
    The unique (up to symmetries) positive radial $H^1$-solution $Q$ of \eqref{nNLS-elliptic-eq-GS}  and exponentially decaying to zero is called  \emph{ground state}. Note that in the one dimensional case,  is explicitly written as 
   
    \[
    Q(x)=\Big[\frac{p+2}{2}\sech^2\big(\tfrac{p}{2} x\big)\Big]^{1/p}
    %\ee
    \]
    and then 
    \be\label{nNLS-Standing-Wave-1D}
    u_{\omega}(t,x)=e^{i\omega  t}Q_{\omega}(x)=e^{i\omega  t}\omega^{1/p}\Big[\frac{p+2}{2}\sech^2\big(\tfrac{p}{2} \sqrt{\omega}x\big)\Big]^{1/p}, \quad \omega>0.
    \ee
    Furthermore, we have that the energy  \eqref{energy} at $U=e^{it}Q$ is
    
    $$E[U]= \frac{np-4}{2(p+2)}\|Q\|_{p+2}^{p+2},$$ and consequently
    \be\label{energy-standing-waves}
    \begin{split}
    	& E[U] < 0 \quad \text{if}\quad 0< p < 4/n,\\
    	& E[U] = 0 \quad \text{if}\quad p=4/n,\\
    	& E[U] > 0 \quad \text{if}\quad 4/n < p < p_n^*.
    \end{split}
    \ee 
     
    \medskip 
    In this work we are interested to characterize nonexistence features of another type of special global solutions of \eqref{nNLS} with a more complex analytical structure, and currently known as \emph{breathers}. Qualitatively these breather solutions are \emph{spatially localized and time periodic solutions} and they are different from standing waves \eqref{nNLS-Standing-Wave-1D}. See Definition \ref{defbreather} for a more precise description of these wave packet like solutions.

    \medskip
    Breather solutions appear in modeling the propagation of pulses in nonlinear media as for gravity waves in the ocean. It is currently believed that they describe rogue or extreme waves in ocean, as well as they appear in the dynamics of optical and matter waves in nonlinear dispersive physical systems. See \cite{Jia},\cite{Dau},\cite{Kip} for further reading. 
    
    \medskip  
     In the context of \eqref{nNLS} with $p=2,~\varepsilon=-1$, i.e. the corresponding \emph{focusing} cubic-NLS equation:
     \be\label{3NLS}
     iu_t + u_{xx} =-|u|^2u,
     \ee
    explicit breather solutions for  \eqref{3NLS} with zero and non-zero BC were found by \cite{SY},\cite{Per},\cite{Kuz},\cite{Ma}. Moreover these cubic NLS breathers were characterized variationally and they satisfy a fourth order nonlinear ODE (see \cite{AFM}).
    To complement this result, we present  mass  \eqref{mass}, energy \eqref{energy} and momentum \eqref{momentum} for some well-known one-dimensional breathers for \eqref{3NLS}. Firstly,  the original Satsuma-Yajima breather solution \cite{SY}-\cite{AFM} (a one dimensional, zero background cubic NLS breather)

    \be\label{BSY}
    u_{SY}(t,x):=\frac{4\sqrt2(\cosh(3x)+3e^{8it}\cosh(x))}{\cosh(4x)+4\cosh(2x)+3\cos(8t)},
    \ee
    \noindent
    with 
    \[
    M[u_{SY}]=16,\quad E[u_{SY}]=-\frac{112}{3},\quad\text{and}\quad P[u_{SY}]=0.
    \]

    \medskip
    Moreover, and for the sake of completeness, in the case of non-zero background solutions, if we have in mind that mass,  energy and momentum are now re-defined as
    
    %\be\label{mass2}
    \[
    M_{nz}[u](t):=\int_{\R}(|u|^2-1)=M_{nz}[u](0),
%    \ee
    \]
   % \be\label{energy2}
   \[
    E_{nz}[u](t):=\int_{\R}(|u_x|^2-\frac12(|u|^2-1)^2)=E_{nz}[u](0),
    \]
    
    \noindent
    and
   
    %\be\label{momentum2}
     \[
    P_{nz}[u](t):=\text{Im}\int_{\R}(\bar{u}-e^{-it})u_x=P_{nz}[u](0),
    \]
    it is easy to see that the Peregrine breather\footnote{Note that although \eqref{BP} is not periodic in time, it is a  limiting case of \eqref{BKM} as $a \rightarrow \frac12^{+}$. Moreover, it satisfies a  fourth order ODE like any other breather solution (see \cite{AFM}), and this is the main reason to name it as breather beside its algebraic character.} \cite{Per}

    \be\label{BP}
    u_{P}(t,x):=e^{it}\left(1-\frac{4(1+2it)}{1+4t^2+2x^2}\right),
    \ee
    \noindent
    and the Kuznetsov-Ma \cite{Kuz}-\cite{Ma} breather
    \be\label{BKM}
    \begin{aligned}
    	&u_{KM}(t,x):=e^{it}\left(1-\sqrt2\beta\frac{\beta^2\cos(\alpha t)+i\alpha\sin(\alpha t)}
    	{\alpha\cosh(\beta x)-\sqrt2\beta\cos(\alpha t)}\right),\\
    	&\alpha:=\sqrt{8a(2a-1)},~~\beta:=\sqrt{2(2a-1)},~~a>\frac{1}{2},
    \end{aligned}
    \ee
    
    \medskip
    \noindent
    satisfy
    \[\begin{aligned}
    	& M_{nz}[u_{P}]=0,\quad E_{nz}[u_{P}]=0,\quad\text{and}\quad P_{nz}[u_{P}]=0,\\
    	& M_{nz}[u_{KM}]=4\beta,\quad E_{nz}[u_{KM}]=-\frac{8}{3}\beta^3,\quad\text{and}\quad P_{nz}[u_{KM}]=0.\\
    \end{aligned}\]   
     For more details about these cubic NLS breathers, see \cite{Akh}, \cite{AFM2}, \cite{Kip},  \cite{Kip2}, \cite{Tsu},\cite{GV1}.
      
    \begin{remark}    
    Note that there are other families of breather solutions for \eqref{3NLS}, which are not in the scope of this work. In particular, the Akhmediev breather (spatially periodic, time localized), corresponding to the case $0<a<\frac{1}{2},$ which is defined as

   % \be\label{BA}
   \[
    \begin{aligned}
    	&u_{A}(t,x):=e^{it}\left(1+\frac{\alpha^2\cosh(\beta t)+i\beta\sinh(\beta t)}
    	{\sqrt{2a}\cos(\alpha x)-\cosh(\beta t)}\right),\\
    	&\beta:=\sqrt{8a(1-2a)},~~\alpha:=\sqrt{2(1-2a)}.
    \end{aligned}
    \]
    \end{remark}

    \begin{remark}
    It is important to highlight that there  are some canonical and non canonical dispersive equations bearing breather solutions in one spatial dimension. These models include some particular powers in the generalized Korteweg-de Vries (gKdV) equation 
    \be\label{gkdv}
    u_t+\partial_x(u_{xx} + u^p)=0
    \ee
    and the sine-Gordon equation
    \be\label{sg}
    \phi_{tt} - \phi_{xx}=-\sin(\phi).
    \ee
    \end{remark}

    \medskip
As we said above, characteristics inherent to the breather concept that will be used in this work are its {\em periodicity in time} and its {\em spatial localization}. In fact, this object can not exist in the defocusing case, as it will be shown in the corresponding case of NLS \eqref{nNLS} (see Th. \ref{Teo1}).

    In the focusing case of NLS model, to distinguish these breather solutions from  ground states \eqref{nNLS-Standing-Wave-1D}, we present a more precise definition of breather for the general case, which will be used through this work. For   $\omega >0$ 
    and being $u$  a solution of \eqref{nNLS}, we first rewrite $u$ in the following way:
   
    \be\label{def-breather-decomp}
    u(t,x)=e^{i\omega t}\big( Q_{\omega}(x) + z(t,x)\big),
    \ee
    with $Q_{\omega}$ satisfying \eqref{nNLS-elliptic-eq-omega}.
    \noindent
    Therefore, $z(t,x)$ is a solution of the corresponding perturbed equation 
   
    \be\label{def-breather-perturbequat}
    iz_t -\omega z + \Delta z =-|Q_{\omega} + z|^p(Q_{\omega}+ z) + |Q_{\omega}|^pQ_{\omega}.
    \ee
   	
   \begin{definition}[{\it Breather}]\label{defbreather} We say that $u$ is a \emph{breather} solution of \eqref{nNLS}, if there exist  $\omega,C,\gamma >0$ and $\mu\in\R$ 
   such that  rewriting it as \footnote{Note that  $\mu=0$ corresponds to the zero background case to be analyzed in Section \ref{s2}.}
   \[
   u(t,x)=e^{i\omega t}\big( Q_{\omega}(x) + z(t,x)\big),
   \]
   \noindent
   we have that $z$ is a nontrivial ($z\neq 0$) solution of \eqref{def-breather-perturbequat} and the next properties hold:
   	
   \begin{enumerate}[(i)]
   	\item {\it periodicity in time:} for all  $(t, x)\in \R\times \R^n$, we have $u(t+2\pi/\omega,x)=u(t,x).$ 
   			
   	\medskip 
   	\item  {\it regularity in spatial variable:} for all $t\in\R$ we have
   	$z(t, \cdot )\in C^{\infty}(\R^n) \cap \big(\mu + H^1(\R^n)\big).$
      
   	\medskip 
   	\item {\it exponential localization:} $\sup\limits_{t\in \R}\big|D^{\alpha}_x \big(z(t, \cdot)-\mu\big)\big|\le Ce^{-\gamma|x|}$ for $|\alpha|\leq2.$ 
    \end{enumerate}
   	\end{definition}
   	
   \medskip 
   It is important to observe that  if $u(t,x)$ is a breather solution of \eqref{nNLS}, then
  
   \begin{align}
   	&u(t,x+x_0),~~\forall\, x_0\in\R,\qquad{\text (translation)}\label{traslation-nNLS}\\
   	&e^{i\theta}u(t,x),~~\forall\, \theta\in\R,~~~~\qquad\qquad{\text (phase)}\\
   	&\lambda^{2/p}u(\lambda^2 t,\lambda x),~~~\forall\, \lambda>0,\qquad{\text (scaling)}\label{scaling-nNLS}
   \end{align}
   are also breather solutions. In \eqref{scaling-nNLS} notice that the rescaled breather has main period $\frac{2\pi}{\omega \lambda^2}$.
  
   \begin{remark}[\textit{Precluding standing waves}]
 	Note that with Definition \ref{defbreather} we still keep  main essential features of breathers, i.e. they still are (i) periodic in time  and (ii) spatially localized solutions but, and equally important, precluding standing waves \eqref{nNLS-Standing-Wave-1D} as breather solutions for \eqref{nNLS} with $\varepsilon =-1$.  We will show in this work that the relation between the sign of energy and the  power $p$ of the nonlinearity given in \eqref{energy-standing-waves} are necessary conditions for the existence of breathers in their respective cases. 
    \end{remark}

    \medskip
    It is currently known that the existence of breather solutions in nonlinear PDEs is a unusual property, coming from the balance between 
    dispersion and nonlinearity of the equation. Previous results about the nonexistence of breather solutions for the gKdV equation \eqref{gkdv} go back to \cite{MP}, where authors presented a complete and rigorous description of different nonlinearities which allow breather solutions and other cases which preclude their existence, as for the classical KdV equation ($p=2$ in \eqref{gkdv}). For a deeper study on existence and nonexistence of  (quasimonochromatic) breathers for gKdV in higher dimensions see \cite{FFMP}. With respect to sine-Gordon equation \eqref{sg} and scalar field equations (like Klein-Gordon and $\phi^6$ equations), some results about existence and nonexistence of breathers can be found in \cite{BMW}, \cite{Denz}, \cite{KMM1},  \cite{Kich}, \cite{Man}, \cite{Vui}.  For NLS models and as far as we know, the only nonexistence result for (small) breathers was presented in \cite{ME} dealing with cubic NLS \eqref{3NLS}.

    \medskip
    Our main aim is to understand how much we can extent the existence (or nonexistence) of breather solutions, which are not standing waves, of the one dimensional cubic NLS \eqref{3NLS} to a more general model as \eqref{nNLS}, in order to draw a road-map to guide us in the search of these special localized profiles. Our inspiration came from a previous approach for KdV type equations in \cite{MP}. In particular, we are interested to describe the role of the nonlinearity and dimension in the existence/nonexistence of NLS breathers in two main cases: with zero ($\mu=0$) and with  nonzero ($\mu\neq0$) background.

    \medskip
    Our nonexistence results for NLS breathers are based on the use of virial type functionals which initially appear in the study of blow-up and singular solutions (which obviously can not be breather solutions) in evolution equations. We show that these virials can be properly used in other settings (not necessarily singularity formation), in particular to preclude existence of breather solutions.

    \medskip
    Finally note that the current work may be useful for advancing to a  better understanding of the Soliton Resolution Conjecture for the NLS \eqref{nNLS} and for related nonlinear dispersive PDEs, once we rigorously justify in which physical regimes of several NLS models no breather-type solutions can be found.
    
   % the long-time behavior of dispersive PDEs, and NLS in particular. Understanding the existence (or nonexistence) of breathers is crucial for advancing towards the Soliton Resolution Conjecture for these equations.
    
   \subsection*{Acknowledgments} Authors   wish to thank Claudio Mu\~noz and many members of the DIM-Univ. Chile for enlightening and constructive discussions about
 this work. We also thank Guillaume Ferriere (INRIA) for his helpful comments, and finally, we are indebted to the referees for their comments and suggestions that actually improved a previous version.

\medskip
{\bf Data Availability Statements:} the data included in this manuscript can be found at arXiv:2408.09862.

\medskip
{\bf Conflict of interest:} all authors declare that they have no conflicts of interest. 

    \medskip 
    \section{Breathers with zero background}\label{s2}
   
     We firstly remember that in the \emph{defocusing} case ($\varepsilon=1$) in \eqref{nNLS}, one-dimensional localized standing wave solutions are not allowed. In fact in this regime standing waves are kink-like solutions, i.e. nonlocalized profiles. Moreover, we prove that even in that case, breather solutions can not exist for \eqref{nNLS} in any dimension. On the other hand, in the  \emph{focusing} case  ($\varepsilon=-1$) in \eqref{nNLS}, we can not expect breather solutions $u$, for instance, when

     \be\label{superCritnNLS}
     E[u]<0\qquad\text{and}\qquad 4/n<p < p_n^*,
     \ee
     
     \noindent 
     with $p_n^{*}$ as in \eqref{Pcritico}. In fact, for initial data satisfying \eqref{superCritnNLS}, it is known the existence of blow-up in finite time, i.e. there exists a positive time $T^{*}<\infty$ such that 
     $$\lim\limits _{t\to T^*}\|\nabla_x u(t,\cdot)\|_{L^2}=\infty.$$
     See \cite[Th. 6.3, p.134]{LP} for further reading.

    \medskip
    The starting point to obtain blow-up results for \eqref{nNLS} in the focusing case ($\varepsilon=-1$) comes from \cite{Glass} based on previous ideas from \cite{ZS}. The main argument rests in the obtention of a suitable {\it virial function,} which allows to study its detailed time evolution. Namely,  the virial function\footnote{We denote$\, 
    	\langle x,\, y\rangle =\sum\limits_{j=1}^{n} \bar{x}_jy_j$, with $x, y\in \C^n$.} (a kind of \emph{weighted momentum})

    \be\label{virial}
    \widetilde{P}[u](t):=\text{Im}\int_{\R^n}\langle x,\bar{u}\nabla_x u\rangle dx,
    \ee

    \medskip
    \noindent
    having that (see e.g. \cite{LP})

    \be\label{virial2}
    \widetilde{P}'[u](t)=2\|\nabla_xu(t,\cdot)\|_{L^2}^2 +\varepsilon \frac{np}{p+2}\|u(t,\cdot)\|^{p+2}_{L^{p+2}}.
    \ee
    Note that the weighted momentum \eqref{virial} is well-defined for $H^1(\R^n)\cap L^2(\R^n,|x|^2dx)$-solutions (see \cite[Prop. 6.1, p.131]{LP}) and breather solutions belong to this functional space\footnote{Note that the same reason applies to other models considered in this work}.

    \medskip
    We note that in fact,  \eqref{virial2} is depending on the energy \eqref{energy} in the following way:

    \be\label{dtvirial}
    \widetilde{P}'[u](t)=2 E[u]  -\varepsilon\frac{4-np}{p+2}\|u(t,\cdot)\|^{p+2}_{L^{p+2}}.
    \ee
    \noindent
    \medskip
    In the forthcoming lines, we will analyze the time evolution of \eqref{virial}, i.e. \eqref{virial2} or \eqref{dtvirial}, to state non-existence results for breathers (Definition \ref{defbreather}) of NLS \eqref{nNLS}. We will consider two types of signs in the nonlinearity, modelling attractive and repulsive potentials. 

    \medskip
    \begin{theorem}[Nonexistence for defocusing  NLS breathers]\label{Teo1}
     Let $\varepsilon=1$. Then breather solutions can not exist in \eqref{nNLS}.
    \end{theorem}
    
    \begin{proof}
    Note that with  $\varepsilon=1,$ from \eqref{virial2}  we have that $\widetilde{P}'[u](t)>0,~~\forall t\in\R;$ implying that $\widetilde{P}[u]$ is monotonic. Therefore, the existence of breathers is not possible in this case.
    \end{proof}

    \subsection{Nonexistence of breathers for the focusing NLS equation \eqref{nNLS}.}

    As pointed out in Remark 1.5 of \cite{ME}, NLS breather solutions (one dimensional) may exist only if their momentum $P[\cdot]$ vanishes. In fact, the same fact holds in any dimension, from the classical relation
    %\be\label{NC-NullMomemntum}
    \[
    \int_{\R^n}  x_j|u|^2dx=\int_{\R^n}  x_j|u_0|^2dx + 2\left(\int_{\R^n}\bar{u}_0\partial_{x_j}u_0dx\right)t,\quad t\in \R\; (1\leq j\leq n). 
    \]
    See for instance,  \cite[(2.7)]{Killip-Visan}. So, null momentum is a necessary condition for the existence of breathers for \eqref{nNLS}.
    
    \medskip
    In the next result, we present other nonexistence conditions for NLS breathers in the case of $P[u] = 0$.

    \begin{theorem}[Nonexistence for focusing NLS breathers]\label{Teo2b}
    Let $\varepsilon=-1$  and $u$ a solution  of \eqref{nNLS} with $P[u]=0$.  Then, $u$ cannot be a breather  in any of the following regimes:

    \medskip 
    	\begin{enumerate}
    	\item $E[u]>0$\, and\, $p\leq 4/n.$
    	\medskip
    	\item $E[u]=0$\, and\, $p\neq 4/n.$
    	\medskip
    	\item $E[u]<0$\, and\, $4/n\le  p < p_n^*.$
       
    	\end{enumerate}
    \end{theorem}
   
    \begin{proof}
    From \eqref{virial} and \eqref{dtvirial} we obtain the regimes \textit{(1)-(2)-(3)} where $\widetilde{P}'[u](t)$ \eqref{dtvirial} has constant sign, which contradicts the periodicity of \eqref{virial}.
    \end{proof}

   \smallskip 
   \begin{remark} 
	We observe the following:
	\begin{itemize}
		\item Theorem \ref{Teo2b} rules out the existence of any localized 
		standing waves.
		
		\medskip
		\item We consider here $p < p_n^*$ because in this case there is a well-known well-posedness theory  in $H^1(\R^n)\cap L^2(\R^n,|x|^2dx)$. However, when $p \ge  p_n^*$, the proof of the Theorem \ref{Teo2b}  still applies for solutions in any subspace of $H^1$ where functionals \eqref{virial} and \eqref{virial2} are  well-defined.
	\end{itemize}
\end{remark} 

    In view of the statements in Theorem \ref{Teo2b} we conclude the next \emph{necessary conditions} for the existence of breather solutions for \eqref{nNLS}:

    \begin{corollary}\label{Corollary-NC-Exist-nNLS}
    Let $\varepsilon=-1$. Breather solutions  $u$  of \eqref{nNLS} can only exist when  $P[u]=0$ and in any of the following regimes:
     
     \medskip 
     \begin{enumerate}
     \item $E[u]<0$\, and\, $p<4/n.$
     
     \medskip 
     \item $E[u]=0$\, and\, $p=4/n.$
     
     \medskip 
     \item $E[u]>0$\, and\, $4/n< p< p_n^*.$
     \end{enumerate}
    \end{corollary}

    \medskip 
    \begin{remark}
    Note that the explicitly known Satsuma-Yajima breather solution \eqref{BSY}  for \eqref{nNLS}, is precisely in regime \text{(1)} of Corollary \ref{Corollary-NC-Exist-nNLS}. See \cite{AFM} for further details.
    \end{remark}
 
	\subsection{Mass subcritical case (one dimension)}
   Now we consider the focusing case of \eqref{nNLS} in  one dimension,  with  $2\le p < 4$. Using the ground state, we will show nonexistence of breathers for \eqref{nNLS} with small mass  and exponential decay slightly stronger than the ground state.  
  
  \begin{theorem}[Nonexistence of breathers for focusing  NLS with $2\le p<4$]\label{Teo2c}
   Let  $\varepsilon=-1$, $n=1$ and $2\le p< 4$. There exists $\epsilon >0$ such that for any initial data $u_0\in H^{\frac{5}{2}+}(\R)$ satisfying 
   %\be\label{Teo2c-01}
   \[
   \|u_0\|_{L^2} <\epsilon,
   \]
   the corresponding $H^{\frac{5}{2}+}(\R)$-solution of \eqref{nNLS} cannot be a breather with frequency  $\omega>0$ having exponential decay
    \be\label{Teo2c-02}
    \big|\p_x^j\big(u(t,x)-e^{i\omega t}Q_{\omega}(x)\big)\big| \lesssim e^{-\gamma|x|},\; \gamma>\sqrt{\omega}\,(j=0,1,2),
    \ee
    for all $t\ge 0$.
  \end{theorem}
\begin{proof}
We note that since $u(t,\cdot)\in H^{\frac{5}{2}+}(\R)$ for all $t\ge 0$, from Sobolev embedding we have that  $u(t,\cdot)\in C^2_{\infty}(\R)$, that is  $u(t,\cdot)\in C^2(\R)$ and $\lim\limits_{|x|\to  \infty}u(t,x)=0$. Furthermore, suppose the $u$ is a breather solution of \eqref{nNLS} with $\varepsilon =-1,~ n=1$ and time-period $2\pi/\omega$ for some $\omega >0$. Then by Definition \ref{defbreather}, 
 \be\label{proof-Teo2c-01}
 u(t,x)= e^{i\omega t}\big(Q_{\omega}(x) + z(t,x) \big),
 \ee
 and $z  \neq 0$ verifying the equation 
 %\begin{equation}\label{proof-Teo2c-02}
\[
  iz_t -\omega z + z_{xx} =-|Q_{\omega} + z|^p(Q_{\omega}+ z) + |Q_{\omega}|^pQ_{\omega}.
 \]
 and the properties in Definition \ref{defbreather}.  Now let $\eta(\cdot) \in C^{\infty}(\R)$ such that 
\[\eta(x) >0\quad \text{and}\quad \eta(x)\lesssim e^{\sqrt{\omega}|x|},\; \forall\, x\in \R.\]
\noindent
 Then, we define the periodic function 
%\be\label{virial-mass-sub}
\[
\psi(t) := \text{Im}\int_{\R}\eta(x) z(t,x)dx,
\]
\noindent
 which is actually well-defined because the exponential growth of $\eta$ equals the decay of the ground state $Q_\omega$ (i.e. rate $\sqrt\omega$ \eqref{nNLS-Standing-Wave-1D}) and $z$ in \eqref{Teo2c-02} decays faster than $Q_\omega$.
\medskip 
Using the hypothesis \eqref{Teo2c-02} and the properties of $\eta,$ we see that $\psi$ is well-defined as well as its derivative, given by 

%\be\label{proof-Teo2c-04}
\[
\begin{split}
\psi'(t) & =  \text{Im}\int_{\R}\eta(x)\big[-i\omega z +i z_{xx} + i|Q_{\omega}+z|^p(Q_{\omega}+z) -i|Q_{\omega}|^pQ_{\omega} \big]dx \\
&=-\int_{\R}\eta |Q_{\omega}|^pQ_{\omega} dx + \text{Re}\int_{\R}(\eta'' -\omega \eta)zdx + \text{Re}\int_{\R}|Q_{\omega}+z|^p(Q_{\omega}+z)\,\eta dx.
\end{split} 
\]
So, selecting now 
$$\eta(x)= \cosh(\sqrt{\omega}x),$$
we have 
\be\label{proof-Teo2c-05}
\psi'(t)=-\int_{\R}\cosh(\sqrt{\omega}x)Q_{\omega}^{p+1}dx + \text{Re}\int_{\R}|Q_{\omega}+z|^p(Q_{\omega}+z)\,\eta dx.
\ee

\medskip 
On the other hand from Corollary \ref{Corollary-NC-Exist-nNLS} we have that $E[u]<0$, then combining this fact with the Gagliardo-Nirenberg inequality  
and the mass conservation law, we get
\be\label{proof-Teo2c-06}
\|u_x(t, \cdot)\|_{L^2}^2  < \|u(t, \cdot)\|_{L^{p+2}}^{p+2} \lesssim_p \|u_x(t, \cdot)\|_{L^2}^{p/2} \|u_0\|_{L^2}^{2+p/2},
\ee
and therefore, 
\be\label{proof-Teo2c-07}
\|u_x(t, \cdot)\|_{L^2} \lesssim_p \|u_0\|_{L^2}^{\frac{4+p}{4-p}},\; \forall\, t\in \R.
\ee

\medskip 
Now, we can estimate the second integral term in \eqref{proof-Teo2c-05}. Using that 
$|u(t,\cdot)\eta|$ is uniformly bounded in time, H\"older's inequality with
$$
\frac{1}{p}=\frac{4}{p^2}\cdot  \frac{1}{2}+ \Big(1- \frac{4}{p^2}\Big)\frac{1}{p+2},\quad (2\le p < 4),
$$
and also the last inequality in \eqref{proof-Teo2c-06} combined with \eqref{proof-Teo2c-07}, one gets

\be\label{proof-Teo2c-08}
\begin{split}
\Big|\text{Re}\int_{\R}|Q_{\omega}+z|^p(Q_{\omega}+z)\,\eta dx\Big|&\lesssim \int_{\R}|u(t, \cdot)|^pdx\\
& \lesssim \|u(t, \cdot)\|_{L^2}^{4/p} \|u(t, \cdot)\|_{L^{p+2}}^{(p^2-4)/p}\\
&=\text{O}\left(\|u_0\|_{L^2}^{\frac{2p}{4-p}}\right).
\end{split}
\ee

\medskip 
Finally, combining \eqref{proof-Teo2c-05} and \eqref{proof-Teo2c-08}, for small enough $\|u_0\|_{L^2}$ we have that 
%\be\label{proof-Teo2c-09}
\[
\psi'(t)\le -\frac12\int_{\R}\cosh(\sqrt{\omega}x)Q_{\omega}^{p+1}dx <0,\; \forall\,\quad t\in \R,
\]
contradicting the periodicity of $\psi$.
\end{proof}

    \begin{table}[h!]
     	\centering
        \begin{tabular}{|c|c|c|c|c|c|}\hline
     	\multicolumn{6}{|c|}{\textbf{Nonexistence of breathers for mass sub-critical:} $\boldsymbol{0< p < 4/n}$} \\
     	\hline
     	 $\boldsymbol{\varepsilon}$ & $\boldsymbol{n}$ & $\boldsymbol{p}$ & $\boldsymbol{P[u]}$ & $\boldsymbol{E[u]}$ & $\boldsymbol{M[u]}$ \\
     	\hline
     	 \cellcolor{gray!25}1 &  & &  &  & \\
     	 \hline
     	 \cellcolor{gray!25} -1&  &  &\cellcolor{gray!25} $\neq 0$ &  &  \\
     	 \hline
     	 \cellcolor{gray!25} -1&  &  & \cellcolor{gray!25} 0 &\cellcolor{gray!25} $\ge 0$ &  \\
     	\hline
     	 \cellcolor{gray!25} -1& \cellcolor{gray!25} 1& \cellcolor{gray!25} $2\le p <4$& \cellcolor{gray!25} 0 & \cellcolor{gray!25} $<0$ 
     	 & \cellcolor{gray!25} Small (+ \emph{exponential decay faster than} $Q_{\omega}$)\\
     	\hline
     \end{tabular}
     \caption{\small Summary results in the mass sub-critical case. Each line represents a non-existence result with its respective conditions (for parameter or functional) marked in gray. The empty cells means that there are no restrictions for the respective parameter or functional.}
    \end{table}
   
   \subsection{Mass critical case}
    If $p=4/n$, from Corollary \ref{Corollary-NC-Exist-nNLS}, we know that breather solutions $u$ can only be found with $E[u]=0$. In this case, also notice that from \eqref{dtvirial}  
    it follows that $\widetilde{P}'[u](t)\equiv 0$ and consequently
    %\be\label{Omega-Constant}
   \[
    \widetilde{P}[u](t)=\widetilde{P}[u](0),\quad \forall \; t\in\R.
   \]
    On the other hand, the variance of solutions, defined by
    \be\label{variance-NLS}
    V(t):=\int_{\R^n}|x|^2|u(t, x)|^2dx,
    \ee
    satisfies 
    $V'(t)=4\widetilde{P}[u](t)$
    (see for instance (6.18) in \cite[Prop.~6.1, p.131]{LP}). Hence, by integration and using the conservation of the weighted momentum in this situation, one obtains the linear function 
    \be\label{var-NLS-Enull-masscritical}
    V(t):=V(0) + 4\widetilde{P}[u](0)\,t.
    \ee
    Using \eqref{var-NLS-Enull-masscritical} and denoting by $Q$ the ground state solution of  
    %\be\label{GState-NLS-masscritical}
    \[
    -\Delta Q+Q-|Q|^{4/n}Q=0,
    \]
    we can establish a  nonexistence result which removes some possibilities of breather existence in Corollary  \ref{Corollary-NC-Exist-nNLS}-{\it(2)}. In fact, it generalizes a small data result in \cite{ME}.

    \begin{proposition}\label{prop1}
    Let $\varepsilon=-1$ and $p=4/n$. If $u$ is a solution of \eqref{nNLS}  with $P[u]=0$ and $E[u]=0$, then $u$ cannot be a breather  in the following cases:

    \medskip 
    \begin{enumerate}
    \item $\widetilde{P}[u](0) \neq 0$
    \medskip 
    \item[]\qquad  or 
    \medskip 
    \item $\|u(t,\cdot)\|_{L^2} \equiv M[u]< \|Q\|_{L^2}$.
    \end{enumerate}
    \end{proposition}

    \begin{proof}
    If $u(t, \cdot)$ is a breather solution of \eqref{nNLS}, then the variance  $V$  \eqref{variance-NLS} is a periodic function, but  from \eqref{var-NLS-Enull-masscritical} this is not possible in the case $\it{(1)}$. On the other hand, it is well known (see \cite{Weins}) that the corresponding solutions to \eqref{nNLS} with initial data verifying  $\it{(2)}$  are globally well defined in time and have positive energy $E[u]>0$. Hence, these global solutions cannot be breathers.
    \end{proof}
   
    \begin{remark}\label{rem2}
    The results in Proposition \ref{prop1} state that in the mass critical situation ($p=4/n$) it is only feasible that breather solutions $u$ exist in the following case:
    $$E[u]=0,\quad \widetilde{P}[u]\equiv 0,\quad P[u] \equiv 0\quad \text{and}\quad M[u]\ge \|Q\|_{L^2}.$$
    Indeed, $e^{it}Q(x)$ is a standing wave solution of \eqref{nNLS} verifying the above conditions. See \cite{Weins} for further reading. Another interesting observation, having in mind {\it(2)},  is that none of global solutions with subcritical mass of  \eqref{nNLS} is a breather solution.
    \end{remark}

\begin{table}[h!]
	\centering
	\begin{tabular}{|c|c|c|c|c|}\hline
		\multicolumn{5}{|c|}{\textbf{Nonexistence of breathers for mass critical:} $\boldsymbol{p=4/n}$} \\
		\hline
		$\boldsymbol{\varepsilon}$  & $\boldsymbol{P[u]}$ & $\boldsymbol{E[u]}$ & $\boldsymbol{M[u]}$ &  $\boldsymbol{\widetilde{P}[u](0)}$\\
		\hline
		\cellcolor{gray!25}  1&   &  &  &  \\
		\hline
		\cellcolor{gray!25} -1&  $\cellcolor{gray!25} \neq 0$ &  &  & \\
		\hline
		\cellcolor{gray!25} -1&  & \cellcolor{gray!25} $\neq 0$ &  & \\
		\hline
		\cellcolor{gray!25} -1& \cellcolor{gray!25} $= 0$ & \cellcolor{gray!25} $= 0$ & \cellcolor{gray!25} $< M[Q]$ & \\
		\hline
		\cellcolor{gray!25} -1&  \cellcolor{gray!25} $= 0$ &  \cellcolor{gray!25} $= 0$ &  & \cellcolor{gray!25} $\neq 0$ \\
		\hline
	\end{tabular}
	\caption{\small Summary results in the mass critical case. Each line represents a non-existence result with its respective conditions (for parameter or functional) marked in gray. The empty cells means that there are no restrictions for the respective parameter or functional.}
\end{table}

    \subsection{Mass super-critical case}
    Now, we are able to approach the mass super-critical situation ($4/n<p < p^*_n$), namely the next result removes some possibilities of breather existence in Corollary \ref{Corollary-NC-Exist-nNLS}-{\it(3)}.

    \medskip 
    Let $p_n^{*}$ be as in  \eqref{Pcritico}, $\frac{4}{n}<p<p_n^*$ and $s_c=\frac{n}{2}-\frac{2}{p}$. To simplify the notation we rename by $Q^*$ 
    instead $Q_{1-s_c}$ the ground state solution of

   %\be\label{GSsol2}
   \[
    -\Delta Q^* + (1-s_c)Q^*-|Q^*|^pQ^*=0.
    \]

    \begin{proposition}\label{prop2}
    Let $\varepsilon = -1$ and $\frac{4}{n}<p<p_n^{*}$, If $P[u]=0$, and $E[u]>0$, then breather solutions $u$ for \eqref{nNLS} can not exist whenever
    \be\label{prop2-a}
    E[u]^{s_c}M[u]^{1-s_c}<E[Q^*]^{s_c}M[Q^*]^{1-s_c}.
    \ee
    \end{proposition}
   
    \begin{proof}
   Assume that there exists a breather $u$. We first observe that 
   there exists a $t_0>0$ such that 
   $$\|\nabla_x u(t_0,\cdot)\|^{s_c}_{L^2}M[u]^{1-s_c}\neq 
   \|\nabla_xQ^*\|^{s_c}_{L^2}M[Q^*]^{1-s_c};$$
   otherwise, the inequality (3.6) in  \cite[Lemma 3.1]{Fang} would give us a contradiction with  condition \eqref{prop2-a}.  
   Therefore, in view of the  time translation invariance of the solutions, we need to analyze the following two cases:

   \medskip 
   \begin{align}
   & \|\nabla_x u(0,\cdot)\|^{s_c}_{L^2}M[u]^{1-s_c}<
   \|\nabla_xQ^*\|^{s_c}_{L^2}M[Q^*]^{1-s_c},\label{prop2-a1}\\
   & \nonumber\\
   &\|\nabla_x u(0,\cdot)\|^{s_c}_{L^2}M[u]^{1-s_c}>
    \|\nabla_x Q^*\|^{s_c}_{L^2}M[Q^*]^{1-s_c}.\label{prop2-a2} 
   \end{align}

\medskip 
From \cite[Th. 7.1]{HR}, conditions in  \eqref{prop2-a}  and \eqref{prop2-a1} imply that \be\label{cond1HR}\quad
\|\nabla_x u\|^{s_c}_{L^2} \|u\|^{1-s_c}_{L^2} \leq \|\nabla_x Q^{*}\|^{s_c}_{L^2} \|Q^{*}\|^{1-s_c}_{L^2}, \quad\forall t\in\R.
\ee

   \medskip
    On the other hand, using the following  Gagliardo-Nirenberg inequality (see for instance \cite[(6.7)]{LP} with $\alpha = p+1$)
    $$
    \|u\|_{L^{p+2}}\le K_{opt}
    \|\nabla_x u\|^{\frac{np}{2(p+2)}}_{L^2} 
    \|u\|^{1-\frac{np}{2(p+2)}}_{L^2}
    $$
   \noindent
   and the energy $E,$ \eqref{energy} we get that

    \be\label{cond2HR}\quad\|\nabla_x u\|_{L^2}^{2}\leq\frac{2}{p+2} K_{opt}^{p+2} \|\nabla_x u\|_{L^2}^{ps_c}\|u\|_{L^2}^{p(1-s_c)}\|\nabla_x u\|_{L ^2}^{2} + E[u], \ee

    \noindent
    with

    \be\label{cond3HR} \quad K_{opt}^{p+2}=\frac{ \|Q^*\|_{L^{p+2}}^{p+2}}{\|\nabla_x Q^*\|_{L^2}^{np/2}\|Q^*\|_{L ^2}^{2-\tfrac{(n-2)p}{2}}}.
    \ee
   \noindent 
   For \eqref{cond3HR} we refer (7.2) in \cite[p. 464]{HR} wit $p+1$ instead of $p$. Now, combining \eqref{cond1HR}, \eqref{cond2HR}, \eqref{cond3HR} and having in mind the value of $s_c$, we obtain

    \be\label{cond4HR} \begin{aligned}\|\nabla_x u\|_{L^2}^{2}&<\frac{2}{p+2} \frac{ \|Q^*\|_{L^{p+2}}^{p+2}\|\nabla_x Q^*\|_{L^2}^{ps_c}}{\|\nabla_x Q^*\|_{L^2}^{np/2}\|Q^*\|_{L ^2}^{2-\tfrac{(n-2)p}{2}}}  \|Q^*\|_{L^2}^{p(1-s_c)}\|\nabla_x u\|_{L ^2}^{2} + E[u]\\
    &\\
    &=\frac{2}{p+2}\frac{ \|Q^*\|_{L^{p+2}}^{p+2}}{\|\nabla_x Q^*\|_{L^2}^{2}} \|\nabla_x u\|_{L ^2}^{2} + E[u].
    \end{aligned}
    \ee

    \medskip
    On the other hand, the ground state $Q^*$ verifies
    \begin{align*}
     & \|\nabla_x Q^*\|_{L^2}^2 - \|Q^*\|_{L^{p+2}}^{p+2} = -(1-s_c)\|Q^*\|_{L^{2}}^{2}\\
     \intertext{and the Pohozaev's identity:}
     & (n-2)\|\nabla_x Q^*\|_{L^{2}}^{2} - \frac{2n}{p+2}\|Q^*\|_{L^{p+2}}^{p+2} = - n(1-s_c)\|Q^*\|_{L^{2}}^{2};
    \end{align*}

    \medskip
    \noindent
    so using these relations we find 
    
    \be\label{cond5HR}\|\nabla_x Q^*\|_{L^2}^{2}= \frac{np}{2(p+2)}\|Q^*\|_{L^{p+2}}^{p+2}.
    \ee
    \noindent
    Therefore, by using \eqref{cond5HR} in \eqref{cond4HR}, we get

    %\be\label{cond6HR} 
    \[
    \begin{aligned}\|\nabla_x u\|_{L^2}^{2}<\frac{2}{p+2}\frac{2(p+2)}{np} \|\nabla_x u\|_{L ^2}^{2} + E[u],
    \end{aligned}
    \]
    \noindent
    namely

    \be\label{cond7HR} \|\nabla_x u\|_{L^2}^{2}<\frac{4}{np} \|\nabla_x u\|_{L ^2}^{2} + E[u],\quad\forall t\in\R.
    \ee
    \noindent
    By using now, \eqref{energy} in \eqref{cond7HR}, we arrive at

    %\be\label{cond8HR} 
    \[
    \frac{2}{p+2}\|u\|_{L^{p+2}}^{p+2}<\frac{4}{np} \|\nabla_x u\|_{L ^2}^{2},
    \]
    \noindent
    or

    \be\label{cond9HR} 
    \frac{np}{p+2}\|u\|_{L^{p+2}}^{p+2}< 2 \|\nabla_x u\|_{L ^2}^{2}.
    \ee
    \noindent
    Finally, substituting \eqref{cond9HR} in \eqref{virial2}, we obtain that

    \[
     \widetilde{P}'[u](t)>0,\quad\forall t\in\R,
    \]
    \noindent
    which gives a contradiction with the existence of breather solutions.

\medskip

Finally, if we are under  conditions \eqref{prop2-a} and \eqref{prop2-a2}, since a breather with zero background naturally belong to the virial space $\big(|x|u\in L^2(\R^n)\big)$, we note that from \cite[Th. 7.1]{HR}, solutions blow up in finite time, and therefore these cannot be breathers.
\end{proof}

\begin{remark}
It is important to observe that in \cite[Th. 7.1]{HR},  blow up in finite time occurs under the hypotheses \eqref{prop2-a} and \eqref{prop2-a2}, assuming initial data belonging to the virial space $\big(|x|u_0\in L^2(\R^n)\big)$ or $u_0$ radial (when $p< \min\{p_n^*, 5\}$ with  $n>1$). Nevertheless, we do not include additional conditions in \eqref{prop2-a2} because a breather with zero background naturally belongs to the virial space in view of its exponential decay.
\end{remark}

\begin{table}[h!]
	\centering
	\begin{tabular}{|c|c|c|c|}\hline
	\multicolumn{4}{|c|}{\textbf{Nonexistence of breathers for mass super-critical:} $\boldsymbol{4/n<p<p_n^{*}}$} \\
	\hline
		$\boldsymbol{\varepsilon}$  & $\boldsymbol{P[u]}$ & $\boldsymbol{E[u]}$ & \textbf{Mass and Energy Balance}\\
	\hline
		 \cellcolor{gray!25} 1 & & & \\
		\hline
		 \cellcolor{gray!25} -1&  \cellcolor{gray!25} $\neq 0$&  & \\
	\hline
		 \cellcolor{gray!25} -1&  &  \cellcolor{gray!25} $\le 0$ &  \\
	\hline
		\cellcolor{gray!25} -1&  \cellcolor{gray!25} $= 0$ &  \cellcolor{gray!25}$> 0$ & \cellcolor{gray!25} 
		   $E[u]^{s_c}M[u]^{1-s_c}<E[Q^*]^{s_c}M[Q^*]^{1-s_c}$\\
	\hline
	\end{tabular}
	\caption{Summary results in the mass super-critical case. Each line represents a non-existence result with its respective conditions (for parameter or functional) marked in gray. The empty cells means that there are no restrictions for the respective parameter or functional.}
\end{table}

    \section{Breathers with nonzero background}

    In the forthcoming analysis it is suitable to introduce the change of variables $u(t,x)=e^{-i\varepsilon t}v(t,x)$ which reduces \eqref{nNLS} to

    \be\label{nvbcNLS}
    iv_t + \Delta v = \varepsilon(|v|^p-1)v,\quad\varepsilon=\pm1.
    \ee
    \noindent
    Note that \eqref{nvbcNLS} is usually known as a generalized Gross-Pitaevskii equation, when solutions behave asymptotically as constant 1 at $\infty$; namely
    $$\lim\limits_{|x|\to +\infty}|u(t,x)|=1,\quad\forall t\geq0.$$
    In particular, we consider here solutions in the space $1+H^1(\R)$.
    \noindent 
    In cases $p=2,4,$ equation \eqref{nvbcNLS} appears in various areas of Physics as nonlinear optics, Bose-Einstein condensates, fluid dynamics, superfluids. See \cite{GR}, \cite{JR},\cite{JPR},\cite{KR1} and \cite{KR2} for further reading.

    \medskip
    Conserved quantities for \eqref{nvbcNLS}, mass and energy and momentum are now re-defined as

   % \be\label{massnz}
   \[
    M_{nz}[v](t):=\int_{\R^n}(|v|^2-1)=M_{nz}[v](0),
    \]

    \medskip

    \be\label{energynz}
    E_{nz}[v](t):=\int_{\R^n}\left(|\nabla_x v|^2-\tfrac{2\varepsilon}{p+2}(1-|v|^{p+2})\right)=E_{nz}[v](0),
    \ee
    \noindent
    and
    %\be\label{momentumnz}
    \[
    P_{nz}[v](t):=\text{Im}\int_{\R^n}(\bar{v}-1)\nabla_x v=P_{nz}[v](0).
    \]
    \noindent
    Just by adapting the virial function in \eqref{virial}, and only dealing with even powers\footnote{Note that $|v|^2-1$ has e.g. physical meaning of density or intensity on a background in optics.} in nonlinearity (i.e.  $p=2q$), we obtain the following:

    \begin{lemma}\label{pnztilde}
    For any solution $v$ of \eqref{nvbcNLS}, if we define
    \be\label{virialnz}
   \widetilde{P}_{nz}[u](t):=\emph{Im}\int_{\R^n}\langle x,(\bar{v}-1)\nabla_x v\rangle dx,
    \ee
    \medskip
    \noindent
    then  with $q=p/2$, we have
    \be\label{dtvirialnz2}
  \widetilde{P}_{nz}'[u](t)=2E_{nz}[v] -\varepsilon n M_{nz}[v] + \varepsilon\big(\tfrac{4}{p+2}-n\big)\int_{\R^n}(1-|v|^{p+2})  - \varepsilon  n\sum_{k=1}^{q}\binom{q}{k}\int_{\R^n}\frac{(|v|^2-1)^{k+1}}{k+1}.
    \ee

    \end{lemma}
    \begin{proof}
    See Appendix \ref{AppenLemm} for a complete proof.
    \end{proof}

\medskip
For the well-posedness of \eqref{nvbcNLS} in space $1+H^1$,  see  \cite[Appendix A]{BS}. The well definition of \eqref{energynz}-\eqref{dtvirialnz2} follows from the standard arguments as in the NLS case.

Note that the fourth term in \eqref{dtvirialnz2} is sign alternating, but even so we can state  the next result for lower powers and dimensions:

    \begin{theorem}[Nonexistence regimes for  NLS breathers with NVBC]\label{Teo7}
    Breather solutions $v$ of \eqref{nvbcNLS} can not exist in any of the following regimes:
    \medskip
    \begin{enumerate}
    	\item Cubic case ($p=2,\,q=1$):
    	\medskip 
    	\begin{itemize}
    		\item[(i)] $n=1$ and\, $\varepsilon\! \cdot\! \big(E_{nz}[v] -  \frac{\varepsilon}{2}M_{nz}[v]\big) \le 0$.
    		\medskip 
    		\item[(ii)]  $n=2$  and\, $E_{nz}[v]\neq 0$.
    	\end{itemize}
    
    	\medskip 
    	\item  Quintic case ($p=4,\,q=2$): 
    	\medskip 
    	\begin{itemize}
    		\item[(i)] $n=1$ and\, 
            $E_{nz}[v]\neq  0$.
    		\medskip 
    		\item[(ii)]  $n=2$  and\, $\varepsilon\! \cdot\! \big(E_{nz}[v] + \varepsilon M_{nz}[v]\big) \ge 0$.
    	\end{itemize}
    \end{enumerate}
    \end{theorem}

   \medskip 
    \begin{remark}
     In the above result, and for the sake of simplicity we do not consider higher powers and dimensions as a consequence of the algebraic complexity of these cases. This fact does not rule out the applicability of the result to  higher power cases.
    \end{remark}

    \begin{proof}
    Again, noting that a monotonic function can not be periodic, when $p=2\, (q=1),$ it turns out that \eqref{dtvirialnz2} becomes
    \be\label{dtvirialnz2-proof-01}
    \widetilde{P}_{nz}'[u](t)=2E_{nz}[v] -\varepsilon n M_{nz}[v] + \varepsilon(1-n)\int_{\R^n}(1-|v|^4)  - \frac{\varepsilon n}{2}\int_{\R^n}(|v|^2-1)^2,
    \ee
    which implies that \eqref{dtvirialnz2-proof-01} does not change sign under conditions \textit{(1)}-(i) and  \textit{(1)}-(ii).

    \medskip
    Now, we consider  the case $p=4\, (q=2)$. When $n=1$, from \eqref{dtvirialnz2} we have 
    \be\label{dtvirialnz2-proof-02}
    \widetilde{P}_{nz}'[u](t)=2E_{nz}[v], 
    \ee
    and consequently   \eqref{dtvirialnz2-proof-02}  does not change sign with the assumptions in \textit{(2)}-(i). The two dimensional case $n=2$ is algebraically more involved 
    and \eqref{dtvirialnz2} can be rewritten in the following way
    \be\label{dtvirialnz2-proof-03}
    \begin{split}
    \widetilde{P}_{nz}'[u](t)&=2E_{nz}[v] -2\varepsilon M_{nz}[v] +  \frac{4\varepsilon}{3}\int_{\R^2}(|v|^6-1) -\frac{2\varepsilon}{3}\int_{\R^2}(|v|^2-1)^3 -2\varepsilon \int_{\R^2}(|v|^2-1)^2\\
    &=2E_{nz}[v] -2\varepsilon M_{nz}[v] +\frac{2\varepsilon}{3}\int_{\R^2}(|v|^2-1)(1+4|v|^2 + |v|^4)  -2\varepsilon \int_{\R^2}(|v|^2-1)^2.
    \end{split}
    \ee
    Then, using that 
    $$1+4|v|^2 + |v|^4 =6 + 4(|v|^2-1) + (|v|^2-1)(|v|^2+1)$$
    and substituting into \eqref{dtvirialnz2-proof-03} one gets
    \be\label{dtvirialnz2-proof-04}
    \widetilde{P}_{nz}'[u](t)=2E_{nz}[v] + 2\varepsilon M_{nz}[v] +\frac{2\varepsilon}{3}\int_{\R^2}(|v|^2-1)^2(|v|^2+1)  + \frac{2\varepsilon}{3} \int_{\R^2}(|v|^2-1)^2.
    \ee
    Hence, conditions in \textit{(2)}-(ii) imply that \eqref{dtvirialnz2-proof-04} preserves the sign. 
    \end{proof}

    \section{Other NLS models}
   We can analyze nonexistence conditions on breathers with zero background in other  closely related NLS models.  In particular, these non-existence conditions will be justified by denying the time periodicity of solutions without taking into account the specific structure of the ground states of each model, unlike what is stated in Theorem  \ref{Teo2c}. For this reason, and for the models to be considered, we give a broader definition of breathers that can include standing waves. Thus, each non-existence condition will also deny the existence of standing waves (weak breathers).
   
   \begin{definition}[{\it Weak breathers}]\label{defwbreather} We say that $u$ is a \emph{breather} solution of a specific NLS model, if there exist  $\tau,C,\gamma$ and $\alpha^*\geq0$ such that the following holds
   
   \begin{enumerate}[(i)]
   	\item {\it periodicity in time:} for all  $(t, x)\in \R\times \R^n$, we have $u(t+\tau,x)=u(t,x).$ 
   			
   	\medskip 
   	\item  {\it regularity in spatial variable:} for all $t\in\R$ we have
   	$u(t, \cdot )\in C^{\infty}(\R^n) \cap  H^\kappa(\R^n),$ with $\kappa$ the specific energy's regularity of each model.
      
   	\medskip 
   	\item {\it exponential localization:} 
   	$\sup\limits_{t\in \R}\big|D^{\alpha}_x u(t, \cdot)\big|\le Ce^{-\gamma|x|}$ for $|\alpha|\leq\alpha^*.$
    \end{enumerate}
   	\end{definition}

    \subsection{Cubic-quintic NLS equation}\label{sec35NLS}
    The NLS equation with  cubic and  quintic nonlinearities is written as
    \be\label{35NLS}
    iu_t + \Delta u  =  \lambda_1|u|^2u - \lambda_2|u|^4u,\quad (t,x)\in\R\times\R^n,\quad \lambda_1\cdot \lambda_2>0,~\lambda_1,\lambda_2\in\R.
    \ee

    \medskip
    \noindent
    Global well-posedness results for the Cauchy problem associated to \eqref{35NLS} in the energy space $H^1(\R^n)$, $n\leq3$, can be seen in \cite{Caze} and \cite{Zhang}. For this NLS model, mass and momentum are the same as in \eqref{mass}-\eqref{momentum}, but its energy is

    %\be\label{E35nls}
    \[
    E_1[u](t):=\int_{\R^n}\left(|\nabla_x u|^2+\tfrac{\lambda_1}{2}|u|^4-\tfrac{\lambda_2}{3}|u|^6\right)dx=E_1[u](0).
    \]

    \medskip 
    Similar to the case of NLS  with zero background, we recall that we have 
    \begin{equation*}
    \frac{d}{dt}\int_{\R^n}  x|u|^2dx=\text{Im}\int_{\R^n}\nabla_xu(t,x)\bar{u}(t,x)dx=2P[u](t)=2P[u](0), 
    \end{equation*}
    so breather solutions must have null momentum ($P[u]=0$).    On the other hand, remembering the virial $\widetilde{P}[u](t)$ \eqref{virial}, we now get for \eqref{35NLS}

    \be\label{dtvirial35}
    \widetilde{P}'[u](t)=2 E_1[u] + \lambda_1\left(\tfrac{n}{2}-1\right)\|u(t,\cdot)\|^{4}_{L^{4}}-\tfrac{2}{3}\lambda_2(n-1)\|u(t,\cdot)\|^{6}_{L^{6}}.
    \ee

    \medskip 
    \noindent 
    From this derivative, we are able to obtain the following characterization of the nonexistence of breather solutions of \eqref{35NLS} (Definition \eqref{defwbreather} with $\kappa=1$):

    \begin{theorem}[Nonexistence in energy-subcritical dimensions]\label{Teo2}
    A solution $u$ of \eqref{35NLS} with $P[u]=0$ cannot be a breather in any of the following regimes:

    \medskip
    \begin{enumerate}
    \item $n=1$ and\, $\lambda_1\!\cdot \!E_1[u]\le 0.$

    \medskip
    \item $n=2$ and\, $\lambda_2\!\cdot \!E_1[u]\le 0.$
    \end{enumerate}
    \end{theorem}
    \begin{proof}
    Again, noting that a monotonic function can not be periodic, from  \eqref{dtvirial35} we obtain the regimes {\it(1)-(2)}.
    \end{proof}

In the following result, we are able to extend this result to higher dimensions, e.g. $n=3$,

    \begin{theorem}[Nonexistence in energy critical dimension]\label{Teo3}
    	Let $n=3$. A solution $u$ of \eqref{35NLS} with $P[u]=0$ cannot be a breather in any of the following cases:
    	\begin{enumerate}
    		\item $\lambda_1<0$ and $E_1[u] > \frac{3\lambda_1^2}{128|\lambda_2|}M[u]$.
    		\medskip  
    		\item $\lambda_1>0$, $E_1[u]<0$ and $\widetilde{P}[u](0) >0$.
    	\end{enumerate}
    \end{theorem}

    \begin{proof}
    If $n=3$ and $\lambda_1<0$ (i.e. $\lambda_2<0$), then from \eqref{dtvirial35} it follows that
    \be\label{proof-Teo3-01}
    \widetilde{P}'[u](t)=2 E_1[u] - \tfrac{|\lambda_1|}{2}\|u(t,\cdot)\|^4_{L^4}+\tfrac{4|\lambda_2|}{3}\|u(t,\cdot)\|^6_{L^6}.
    \ee
    Now, byCauchy-Schwarz and Young  inequalities we have  
    \be\label{proof-Teo3-02}
    \begin{split}
    \|u(t,\cdot)\|^4_{L^4}& \le \|u(t,\cdot)\|^3_{L^6}M[u]^{1/2}\\
    &\le \tfrac{1}{2\delta^2}\|u(t,\cdot)\|^6_{L^6} + \tfrac{\delta^2}{2}M[u],
    \end{split}
    \ee
    for any $\delta >0$. Thus, choosing $\delta^2 = \frac{3|\lambda_1|}{16|\lambda_2|}$ and using \eqref{proof-Teo3-02} in \eqref{proof-Teo3-01} one gets
   % \be\label{proof-Teo3-03}
   \[
    \widetilde{P}'[u](t) \ge 2 E_1[u] - \tfrac{3\lambda_1^2}{64|\lambda_2|}M[u],
   \]
    which under hypothesis {\it (1)} give us that $\widetilde{P}'[u](t)>0$ for all time, so $u$ cannot be a breather solution. 

    \medskip
    In case {\it(2)}, all solutions blow-up in finite time according to  \cite[Th.1.3, p.425]{Zhang}.
    \end{proof}

    \subsection{Biharmonic  NLS equation}\label{sec4thNLS}
    We also consider the  biharmonic NLS equation

    \be\label{4thNLS}
    iu_t + \mu\Delta u -\Delta^2u=  \varepsilon|u|^{p}u,\quad (t,x)\in\R\times\R^n,
    \ee
    \noindent
    where $\mu\in\R,~\varepsilon=\pm1$ and $0<p < p^{**}_n$, with
    $$
    p^{**}_n=
    \begin{cases}
     \infty& \text{if}\; n\le 4,\medskip \\
     \frac{8}{n-4}& \text{if}\; n\ge 5.
    \end{cases}
    $$

    This model physically allows to describe “small fourth-order dispersion” in the propagation of laser beams in mediums with Kerr nonlinearity \cite{Karp1}-\cite{Karp2}. About the Cauchy problem for \eqref{4thNLS}, see \cite{AKS} and  \cite{Paus}.

    \medskip
    For this higher order dispersive NLS model, mass and momentum are the same as in \eqref{mass}-\eqref{momentum}, but its energy is as follows

  %  \be\label{E4xnls}
  \[
    E_2[u](t):=\int_\R\left(|\Delta u|^2+\mu|\nabla_x u|^2+\varepsilon\frac{2}{p+2}|u|^{p+2}\right)dx.
    \]

    \medskip
    For the virial $\widetilde{P}[u](t)$ \eqref{virial}, we now get for \eqref{4thNLS}

    \be\label{dtvirial4th}
    \begin{split}
    \widetilde{P}'[u](t)&= 4\|\Delta u(t,\cdot)\|^{2}_{L^{2}} +2\mu\|\nabla_x u(t,\cdot)\|^{2}_{L^{2}} + \varepsilon \frac{np}{p+2}\|u(t,\cdot)\|^{p+2}_{L^{p+2}} \\
    &=4E_2[u]-2\mu\|\nabla_x u(t,\cdot)\|^{2}_{L^{2}} + \varepsilon \left(\frac{np-8}{p+2}\right)\|u(t,\cdot)\|^{p+2}_{L^{p+2}}.
    \end{split}
    \ee

In fact, this expression can be seen as a generalization of the case $\mu=0$ in \cite{BL}. From \eqref{dtvirial4th}, we are able to obtain the following characterization of the nonexistence of breather solutions of \eqref{4thNLS} (Definition \eqref{defwbreather} with $\kappa=2$):

    \begin{theorem}[Nonexistence regimes for biharmonic NLS breathers]\label{Teo4}
    Breather solutions $u$ of \eqref{4thNLS} can not exist in any of the following regimes:
    
    \medskip
    \begin{enumerate}
    	
    \item  $\varepsilon=1$ (\textit{defocusing case}) and $\mu \geq0$.

     \medskip
     \item  $\varepsilon=-1$ (\textit{focusing case})
     \begin{enumerate}
         \item[(i)] $p<\frac{8}{n}<p_n^{**}:\quad \mu \leq0~\text{and}~ E_2[u]\geq0.$
         
         \item[(ii)]$\frac{8}{n}<p<p_n^{**}:\quad \mu \geq0~\text{and}~ E_2[u]\leq0.$
     \end{enumerate}

      \medskip
      \item  $\varepsilon=\pm1$ (\textit{def/focusing  cases})~
      \text{with}~~$p=\frac{8}{n},~ (E_2[u],\mu)\neq(0,0)~\text{and}~ \mu E_2[u]\leq0.$
    \end{enumerate}
    \end{theorem}

    \begin{proof}
    Again, noting that a monotonic function can not be periodic, from  \eqref{dtvirial4th} we obtain  regimes {\it(1)}, {\it(2)} and {\it(3)}. In fact, in regime {\it(1)} this is an immediate consequence of first line \eqref{dtvirial4th}. With respect to regimes {\it(2)} and {\it(3)}
    follow directly from second line \eqref{dtvirial4th}.
    \end{proof}

    \subsection{Derivative  NLS equation}\label{Section-CMDNLS}
    We also consider the  one dimensional derivative NLS equation
    \be\label{DNLS}
    iu_t + u_{xx} - i\varepsilon(|u|^2u)_x  = 0,\quad (t,x)\in\R\times\R\quad  (\varepsilon=\pm 1).
    \ee
    \noindent
    For this  dispersive NLS model, and as a completely integrable system, mass is the same as in \eqref{mass}.  Another  conserved quantity corresponds to  a kind of momentum (see (1.2) in \cite{Killip-Netekoume-Visan}), which is written as
 
   \[
    \begin{split}
    H[u](t) &:=\int_\R\left(\frac{i}{2}(\bar{u}u_x-u\bar{u}_{x})+\frac{\varepsilon}{2}|u|^4\right)dx\\
    	&= -\text{Im}\int_\R \bar{u}u_x +\frac{\varepsilon}{2}\int_\R|u|^4\\
        &= H[u](0).
    \end{split}
    \]
    
    \noindent
    Namely,
    %\be\label{Ednls}
    \[
    -2H[u]=2\text{Im}\left( \int_\R\bar{u}u_x\right) -\varepsilon \|u\|^4_{L^4}.
    \]
    Finally there exists a kind of energy
    \be\label{Ednls02}
    E_3[u](t):=\int_\R\left(|u_x|^2+\frac{3i}{4}|u|^2(u\bar{u}_{x}-\bar{u}u_x)+\frac12|u|^6\right)dx = E_3[u](0).\\
    \ee
 
    Now, considering the following virial function:
    %\be\label{virialDNLS}
    \[
    \widetilde{V}(t):=\int_\R x|u|^2,
    \]
    \noindent
    we get that
    \be\label{dtvirialDNLS}
    \begin{aligned}
    	\widetilde{V}'(t)&=\partial_t\int_\R x|u|^2=-\int_\R\partial_x\left(2\text{Im}(\bar{u}u_x) - \varepsilon\frac32\int_\R |u|^4\right)x dx\\
    	&=\int_\R\left(2\text{Im}(\bar{u}u_x) -\varepsilon \frac32\int_\R |u|^4\right) dx\\
    	&=-2H[u] - \frac12\varepsilon\|u\|^4_{L^4}.
    \end{aligned}\ee
    \noindent
    Using the expression above, we get the following result for breathers of \eqref{DNLS} (Definition \eqref{defwbreather} with $\kappa=1$):

    \begin{theorem}[Nonexistence regimes for derivative NLS breathers]\label{Teo5}
    	Breather solutions $u$ of \eqref{DNLS} can not exist in any of the following regimes:
    	\medskip
    	\begin{enumerate}
    		\item  $\varepsilon=-1$ (\it focusing case)  and  $H[u]\le 0$.
              
    		\medskip
    		\item   $\varepsilon=1$ (\it defocusing case)  and  $H[u]\ge 0$.
    	\end{enumerate}
    \end{theorem}
    \begin{proof}
    	Again, noting that a monotonic function can not be periodic, from  \eqref{dtvirialDNLS} we obtain the regimes {\it(1)} and {\it(2)}.
    \end{proof}

\begin{remark}
It is important to highlight that  breathers solutions cannot exist with even or odd spatial symmetry (real and imaginary part with the same parity) to \eqref{DNLS}. Indeed, 
more generally, if $u(t, x)$ is a regular solution  of \eqref{DNLS} such that 
$$u(t, x)=u(t,-x), \quad t\in \R,$$
then $u$ also satisfies the equation
\be\label{DNLS-Remark}
iu_t + u_{xx} + i\varepsilon(|u|^2u)_x  = 0,\quad (t,x)\in\R\times\R.
\ee
Thus, in view of  \eqref{DNLS} and \eqref{DNLS-Remark} we would have that $u$ is a solution of the linear equation 
$$iu_t + u_{xx} =0,$$
which is a contradiction in the context of nontrivial solutions. 
\end{remark}

  \subsection{Logarithmic NLS equation}\label{logNLS-section}

 Finally, we approach the log-NLS equation
    \be\label{logNLS}
    iu_t + \Delta u   =  \varepsilon\log(|u|^2)u,\quad (t,x)\in\R\times\R^n,\quad \varepsilon=\pm1.
    \ee
    \noindent
    This model was introduced in  wave mechanics and nonlinear optics \cite{BBM} (see also \cite{HeRe},\cite{KEB}).

    \medskip
     For \eqref{logNLS} mass and momentum are conserved and equal to \eqref{mass} and \eqref{momentum}. 
     Its energy is written as
    \be\label{Elognls}
    E_4[u](t) :=\int_\R\left(|u_x|^2+2\varepsilon|u|^2\log(|u|^2)\right)dx = E_4[u](0).
    \ee

     Note moreover that in \eqref{logNLS} the nonlinear $log$ term has a singularity at the origin, removing the local Lipschitz continuity.  
    \medskip
    The Cauchy problem for \eqref{logNLS} was fully described  in \cite[Th.1.1]{HO} (see also \cite{CG},\cite{GLN}): namely for initial data in a subspace 
    $$W^1(\R^n):= \{ v \in H^1 (\mathbb{R}^n), |v|^2 \ln{|v|^2} \in L^1 (\mathbb{R}^n) \}\subset H^1(\R^n),$$
     there exists a unique solution $u\in (C\cap L^\infty)(\R,W^1)$ in the case $\varepsilon=-1$, and $u\in C(\R,W^1)$ in the case $\varepsilon=1$.
 
   \medskip
   With respect to explicit solutions of \eqref{logNLS},    standing wave solutions $u(t,x)=e^{i\omega t}G_\omega(x)$ are known, where   
\[
G_\omega(x):=e^{\frac{\omega+n}{2}-\frac{\varepsilon|x|^2}{2}},\quad x\in\R^n,
   \]
   \noindent
   is the so called {\it Gausson}. Beside standing waves, breather solutions are also explicitly known (see \cite{Ferr20}, \cite{BBM}) in the case $\varepsilon=-1,n=1$, namely

     \begin{equation*}
        u^{\alpha} (t,x) := \sqrt{\frac{\alpha_r}{r_\alpha (t)}} \exp \Bigl[ \frac{1}{2} - i \Phi^\alpha (t) - \frac{x^2}{4 r_\alpha (t)^2} + i \frac{\dot{r}_\alpha (t)}{r_\alpha (t)} \frac{x^2}{4} \Bigr], \qquad t,x \in \mathbb{R},
    \end{equation*}
    with
    \begin{equation*}
        \Phi^\alpha (t) := \frac{1}{2} \int_0^t \frac{1}{r_\alpha (s)^2} \diff s + \varepsilon \int_0^t \ln \frac{r_\alpha (s)}{\alpha_0} \diff s - \varepsilon t,
    \end{equation*}
\noindent
and where for any $\alpha \in \mathbb{C}^+$, the function $r_\alpha$ is defined by
    \begin{equation*}
        \ddot{r}_\alpha = \frac{1}{r_\alpha^3} - \frac{2 \varepsilon}{r_\alpha}, \qquad r_\alpha (0) = \text{Re}~\alpha := \alpha_r, \qquad \dot{r}_\alpha (0) = \text{Im}~\alpha := \alpha_i.
    \end{equation*}
   
   \medskip
   \noindent
   In fact the above ODE  has a unique solution $r_{\alpha} \in \mathcal{C}^\infty (\mathbb{R})$ with values in $(0, \infty)$, and it is periodic.
   
   \medskip
   With respect to the virial $\widetilde{P}[u](t)$ \eqref{virial}, we now get for \eqref{logNLS} (see Appendix \ref{App2} for further details)

        \be\label{dtviriallog}
        \widetilde{P}'[u](t)=2\|\nabla_x u\|_{L^2}^2+\varepsilon n M[u].
        \ee

    Now, we characterize  the nonexistence of breather solutions for \eqref{logNLS} (Definition \eqref{defwbreather} with $\kappa=1$):

   \begin{theorem}[Nonexistence regimes for log-NLS breathers]\label{Teo6}
    Let $\varepsilon = 1.$ Then breather solutions can not exist in \eqref{logNLS}.   
    \end{theorem}

   \begin{proof}
    Assuming $\varepsilon=1$, if $u$ is a breather solution of \eqref{logNLS},  $\widetilde{P}[u]$ is a periodic function. But from  \eqref{dtviriallog}, we have that 
     $\widetilde{P}'[u](t)>0,~~\forall t\in\R,$ and we conclude. 
   \end{proof}

\begin{remark}\label{remLog1}
 Note that Theorem \ref{Teo6}  agrees with \cite{CG}, confirming   that there can be no breather solutions in the defocusing case for \eqref{logNLS}.
\end{remark}

\begin{remark}(Inhomogeneous models)\label{Inhom}
The same approach already used to get regimes of nonexistence of breathers for \eqref{nNLS} can be applied to other closely related NLS models. For instance,  that
is the case of the  inhomogeneous Schr\"odinger equation (see \cite{Far}),  written as

   %  \be\label{iNLS}
   \[
     iu_t + \Delta u + \frac{1}{|x|^b}|u|^{2\sigma}u  = 0,\quad (t,x)\in[0,\infty)\times\R^n.
    \]

\noindent   
And even more, another interesting NLS model, is
  %  \be\label{newiNLS}
  \[
     iu_t + \Delta u + k(x)|u|^{2\sigma}u  = 0,\quad (t,x)\in[0,\infty)\times\R^2,
     \]
   \noindent
   for some smooth bounded inhomogeneity $k:\R^2\rightarrow\R^{*}_+$, see \cite{RS} for further details. 
\end{remark}

%%%%%%%%%%%%%%%%%%%%%%%%%%%%%%%
%APENDICE
%%%%%%%%%%%%%%%%%%%%%%%%%%%%%%%

   \appendix
   
   \section{Proof of Lemma \ref{pnztilde}}\label{AppenLemm}
    Considering \eqref{nvbcNLS}, we will multiply it by $2r\partial_r(\bar{v}-1),$ integrate in $x$ and take real part:
    \be\label{appennvbcNLS}
    \begin{aligned}
    &Re\left(2i\int_{\R^n}r\partial_r(\bar{v}-1)v_t\right)=\underbrace{i\int_{\R^n}\left(r\partial_r(\bar{v}-1)v_t - r\partial_r(v-1)\bar{v}_t \right)}_{I}\\
    &=-\underbrace{2Re \int_{\R^n}r\partial_r(\bar{v}-1)\Delta v}_{II} - \underbrace{2\varepsilon Re\int_{\R^n}r\partial_r(\bar{v}-1)(1-|v|^p)v}_{III}.
    \end{aligned}
    \ee
    \noindent
    We will analyse each of the three terms I-II-III in \eqref{appennvbcNLS}. For the first term, we have\footnote{We will consider here the notation $r\partial_r\equiv x_j\partial_{x_j}$ and we will also use the Einstein summation convention.}
    \[
    \begin{aligned}
    &\text{I}=i\int_{\R^n}\left(r\partial_r(\bar{v}-1)v_t - r\partial_r(v-1)\bar{v}_t \right)=
    i\int_{\R^n}\left(x_j\partial_{x_j}(\bar{v}-1)v_t - x_j\partial_{x_j}(v-1)\bar{v}_t \right)\\
    &=i\int_{\R^n}x_j\left(\bar{v}_{x_j}v_t - v_{x_j}\bar{v}_t \right) =
    i\int_{\R^n}x_j\left(\bar{v}_{x_j}v_t + \partial_t\bar{v}_{x_j}v -\partial_t\bar{v}_{x_j}v  - \bar{v}_tv_{x_j} \right)\\
    &=i\int_{\R^n}x_j\left(\partial_{t}(\bar{v}_{x_j}v) - \partial_{x_j}(\bar{v}_tv) \right) = 
    i\frac{d}{dt}\int_{\R^n}rv\partial_{r}\bar{v} - i\int_{\R^n}x_j\partial_{x_j}(\bar{v}_tv) \\
    &=i\frac{d}{dt}\int_{\R^n}rv\partial_{r}\bar{v} + n\int_{\R^n}i\bar{v}_tv\\
    &=i\frac{d}{dt}\int_{\R^n}rv\partial_{r}\bar{v} + n\int_{\R^n}\left(\Delta \bar{v}+\varepsilon(1-|v|^p)\bar{v}\right)v\\
    &=i\frac{d}{dt}\int_{\R^n}rv\partial_{r}\bar{v} - n\int_{\R^n}|\nabla_x v|^2 +\varepsilon n\int_{\R^n}(1-|v|^p)|v|^2.
    \end{aligned}
    \]	
    \noindent
    For the second term, we get
    \[
    \begin{aligned}
    &\text{II}=2Re \int_{\R^n}r\partial_r(\bar{v}-1)\Delta v = 2Re \int_{\R^n}x_j\partial_{x_j}(\bar{v}-1)\partial_{x_k}^2 v
    =(n-2) \int_{\R^n}|\nabla_x v|^2.
    \end{aligned}
    \]	
    \noindent
    Finally, for the third term,
    \[
    \begin{aligned}
    &\text{III}=2\varepsilon Re\int_{\R^n}r\partial_r(\bar{v}-1)(1-|v|^p)v\\
    &= 2\varepsilon Re\int_{\R^n}x_j\partial_{x_j}(\bar{v}-1)(1-|v|^p)v = 2\varepsilon Re\int_{\R^n}x_j\bar{v}_{x_j}(1-|v|^p)v \\
    & = \varepsilon \int_{\R^n}x_j(1-|v|^p)(\bar{v}_{x_j}v+v_{x_j}\bar{v})  =  \varepsilon \int_{\R^n}x_j(1-|v|^p)\partial_{x_j}(|v|^2)\\
    &=\varepsilon \int_{\R^n}x_j(1-(|v|^2)^{q})\partial_{x_j}(|v|^2)\big|_{q=p/2} = 
    \varepsilon \int_{\R^n}x_j(1-(|v|^2-1+1)^{q})\partial_{x_j}(|v|^2)\\
    &=\varepsilon \int_{\R^n}x_j(1-\sum_{k=0}^{q}c_k(|v|^2-1)^{k})\partial_{x_j}(|v|^2) = 
    -\varepsilon \int_{\R^n}x_j(\sum_{k=1}^{q}c_k(|v|^2-1)^{k})\partial_{x_j}(|v|^2)\\
    &=-\varepsilon \int_{\R^n}x_j\sum_{k=1}^{q}c_k\partial_{x_j}\left(\frac{(|v|^2-1)^{k+1}}{k+1}\right) =
    \varepsilon  n\int_{\R^n}\sum_{k=1}^{q}c_k\frac{(|v|^2-1)^{k+1}}{k+1}\\
    &=\varepsilon  n\sum_{k=1}^{q}c_k\int_{\R^n}\frac{(|v|^2-1)^{k+1}}{k+1},\quad c_k=\begin{pmatrix}q\\k\end{pmatrix}.
    \end{aligned}
    \]	

    \medskip
    Gathering I, II and III, and writing down \eqref{appennvbcNLS} as
    \[
     \text{I} = -\text{II} - \text{III},
    \]
    \noindent
    we get
    \[
     \begin{aligned}
      &i\frac{d}{dt}\int_{\R^n}rv\partial_{r}\bar{v} - n\int_{\R^n}|\nabla_x v|^2 +\varepsilon n\int_{\R^n}(1-|v|^p)|v|^2\\
      &=-(n-2) \int_{\R^n}|\nabla_x v|^2 - \varepsilon  n\sum_{k=1}^{q}c_k\int_{\R^n}\frac{(|v|^2-1)^{k+1}}{k+1},
     \end{aligned}
    \]
    \noindent
    and therefore
    \[
     \begin{aligned}
      &\widetilde{P}_{nz}'[u](t)=  (n-n+2)\int_{\R^n}|\nabla_x v|^2 -\varepsilon n\int_{\R^n}(1-|v|^p)|v|^2- \varepsilon  n\sum_{k=1}^{q}c_k\int_{\R^n}\frac{(|v|^2-1)^{k+1}}{k+1},
     \end{aligned}
    \]
    \noindent
    and from \eqref{energynz} $E_{nz}[v]=\int_{\R^n}|\nabla_x v|^2-\frac{2\varepsilon}{p+2}(1-|v|^{p+2})$, we get

    \[
     \begin{aligned}
      &\widetilde{P}_{nz}'[u](t)=  2E_{nz}[v] + \varepsilon\left(\frac{4}{p+2}\right)\int_{\R^n}(1-|v|^{p+2}) -\varepsilon n\int_{\R^n}(1-|v|^p)|v|^2 - \varepsilon  n\sum_{k=1}^{q}c_k\int_{\R^n}\frac{(|v|^2-1)^{k+1}}{k+1}\\
      &=2E_{nz}[v] + \varepsilon\left(\frac{4}{p+2}\right)\int_{\R^n}(1-|v|^{p+2}) -\varepsilon n\int_{\R^n}(|v|^2-|v|^{p+2}) - \varepsilon  n\sum_{k=1}^{q}c_k\int_{\R^n}\frac{(|v|^2-1)^{k+1}}{k+1}\\
      &=2E_{nz}[v] + \varepsilon\left(\frac{4}{p+2}\right)\int_{\R^n}(1-|v|^{p+2}) -\varepsilon n\int_{\R^n}(|v|^2-1+1-|v|^{p+2}) - \varepsilon  n\sum_{k=1}^{q}c_k\int_{\R^n}\frac{(|v|^2-1)^{k+1}}{k+1}\\
      &=2E_{nz}[v] -\varepsilon n M_{nz}[v] + \varepsilon\left(\frac{4}{p+2}-n\right)\int_{\R^n}(1-|v|^{p+2})  - \varepsilon  n\sum_{k=1}^{q}c_k\int_{\R^n}\frac{(|v|^2-1)^{k+1}}{k+1}.
     \end{aligned}
    \]

\section{Proof of \eqref{dtviriallog}}\label{App2}

Following the same steps that for Lemma \ref{pnztilde}, we will multiply \eqref{logNLS} by $2r\partial_r\bar{u},$ integrate in $x$ and take real part:
    \be\label{appenPlog}
\begin{aligned}
&Re\left(2i\int_{\R^n}r\partial_r\bar{u}u_t\right)=\underbrace{i\int_{\R^n}\left(r\partial_r\bar{u}u_t - r\partial_ru\bar{u}_t \right)}_{I}\\
    &=-\underbrace{2Re \int_{\R^n}r\partial_r\bar{u}\Delta u}_{II} + \underbrace{2\varepsilon Re\int_{\R^n}r\partial_r\bar{u}\log(|u|^2)u}_{III}.
    \end{aligned}
     \ee

For the first term, we have
    \[
    \begin{aligned}
    &\text{I}=i\int_{\R^n}\left(r\partial_r\bar{u}u_t - r\partial_ru\bar{u}_t \right)
    =i\frac{d}{dt}\int_{\R^n}ru\partial_{r}\bar{u} - n\int_{\R^n}|\nabla_x u|^2 -\varepsilon n\int_{\R^n}\log(|u|^2)|u|^2.
    \end{aligned}
    \]	
    \noindent
    For the second term, we get
    \[
    \begin{aligned}
    &\text{II}
    =(n-2) \int_{\R^n}|\nabla_x u|^2.
    \end{aligned}
    \]	
    \noindent
    And for the third term,
    \[
    \begin{aligned}
    &\text{III}=2\varepsilon Re\int_{\R^n}r\partial_r\bar{u}\log(|u|^2)u
    =  n\varepsilon\int_{\R^n}|u|^2 - \varepsilon n\int_{\R^n}\log(|u|^2)|u|^2.
    \end{aligned}
    \]

\medskip
Now,  collecting I, II and III above and writing down \eqref{appenPlog} as
    \[
     \text{I} = -\text{II} + \text{III},
    \]
    \noindent
    we get

    \[\widetilde{P}'[u](t)=2\|\nabla_x u\|_{L^2}^2+\varepsilon n M[u].\]

    \end{document}